\documentclass[12pt]{extarticle}
\usepackage{amsmath,amsthm,latexsym,graphicx,epstopdf,cite,amssymb,authblk,float,subcaption,cleveref}
\usepackage[margin=3cm]{geometry}

\newcommand\blfootnote[1]{
	\begingroup\renewcommand\thefootnote{}\footnote{#1}
	\addtocounter{footnote}{-1}
	\endgroup}

\newtheorem{theorem}{Theorem}[section]
\newtheorem{lemma}{Lemma}[section]

\numberwithin{equation}{section}

\theoremstyle{definition}

\newtheorem{example}[theorem]{Example}
\usepackage[T1]{fontenc}
\usepackage[utf8]{inputenc}
\usepackage{authblk}

\title {\bf Strong metric dimension of generalized Jahangir graph}
\author{Rashid Farooq\thanks{Corresponding author.} }
\author{Naila Mehreen}
\affil{School of Natural Sciences,
	National University of Sciences and Technology,
	H-12 Islamabad, Pakistan}

\date{}
\begin{document}
	\maketitle
	\date{}
	\blfootnote{\raggedright Email addresses: farook.ra@gmail.com (R. Farooq), nailamehreen@gmail.com (N. Mehreen).}
\begin{abstract}
	Let $G$ be a simple and connected graph
	with vertex set $V(G)$. A vertex $w\in V(G)$ strongly resolves two vertices $u,v \in V(G)$ if 
	$v$ belongs to a shortest $u-w$ path or
	$u$ belongs to a shortest $v-w$ path. A set $W \subseteq V(G)$ is a strong resolving set for 
	$G$ if every pair of vertices of $G$ is strongly resolved by some vertex of $W$. A strong
	metric basis of $G$ is a strong resolving set for $G$ with minimum cardinality. 
	The strong metric dimension of $G$, denoted
	by $sdim(G)$, is the cardinality of a strong metric basis of $G$. 
	In this paper we compute the strong metric dimension of generalized 
	Jahangir graph $J(n,m)$, where $m\geq 3$ and $n\geq 2$.
\end{abstract}
\begin{quote}
	{\bf Keywords:}
	Strong metric dimension, Strong resolving set, 
	Mutually maximally distant vertices, Jahangir graph.
\end{quote}
\begin{quote}
	{\bf AMS Classification:} 
	05C12  
\end{quote}

\section{Introduction}\label{sec-intro}
Let $G$ be a simple and connected graph with vertex set $V(G)$ and edge set $E(G)$. 
If two vertices $u,v\in V(G)$ are endpoints of an edge then we represent the edge by $uv$.
The distance between two vertices $u,v\in V(G)$, denoted by $d_G(u,v)$, is the length of a 
shortest $u-v$ path in $G$. 
The diameter $d(G)$ of $G$ is defined by 
$d(G) = \max\{d_G(u,v)\mid u,v\in V(G)\}$. 
The set of all vertices that are adjacent to $u$ in $G$ 
is called neighborhood of $u$ and is denoted by $N(u)$. 
The degree $d_G(v)$ of a vertex $v$ in $G$ is the number of 
edges incident to $v$.
If there is no risk of confusion, we write distance between vertices $u,v\in V(G)$ 
by $d(u,v)$. Similarly, the degree of $v$ in $G$ is written by $d(v)$ for short. 
A cycle of length $n$ is denoted by $C_n$
and a path of length $n-1$ is denoted by $P_n$.
A complete graph of order $n$ is denoted by $K_n$.

A vertex $w\in V(G)$ resolves two vertices $u,v\in V(G)$ 
if $d(u,w)\neq d(v,w)$. Suppose $W=\{w_1,w_2,\dots,w_k\}\subseteq V(G)$ is an ordered set of vertices.
The representation of a vertex $v\in V(G)$ with 
respect to $W$, denoted by $r_G(v|W)$,  is the $k$-vector $r_G(v|W)=(d(v,w_1),\dots,d(v,w_k))$. 
Then $W$ is a resolving set of $G$ if $r_G(v|W)\neq r_G(u|W)$ for every two distinct vertices $u,v\in V(G)$. 
That is, $W$ is a resolving set of $G$ if every two distinct vertices of $G$ are resolved by some vertex of $W$.
A resolving set of $G$ with minimum cardinality over all resolving sets of $G$ is called a basis of $G$. 
The cardinality of a basis of $G$ is called the metric dimension of $G$ and is denoted by $dim(G)$. 
The concept of metric dimension was introduced by Slater \cite{slater1975, slater1988}.
Harary and Melter \cite{HM1976} independently studied the metric dimension.
For further study and applications of metric dimension, we refer 
\cite{Imran, CHMPP2012, CHMPPS2005, CGH2008, CEJO2000, JS1994, johnson1993, KRR1996, LKRJ2018,MT1984, tomescu2008, YI2015} to the reader.

In 2004, Seb\H{o} and Tannier \cite{ST2004} introduced the concept of strong metric dimension of graphs 
which is more restricted invariant than the metric dimension of graphs. 
A vertex $w\in V(G)$ strongly resolves two vertices $u,v\in V(G)$ if 
either $d(u,w)=d(u,v)+d(v,w)$ or $d(v,w)=d(v,u)+d(u,w)$. 
Alternatively, $w$ strongly resolves $u$ and $v$ if 
either $v$ belongs to a shortest $u-w$ path or $u$ belongs to a 
shortest $v-w$ path. A vertex set $W\subseteq V(G)$ is a strong 
resolving set of $G$ if every two distinct vertices of $G$ are 
strongly resolved by some vertex of $W$. A strong resolving set
of $G$ with minimum cardinality over all strong resolving
sets of $G$ is called a strong metric basis of $G$. 
The strong metric dimension of $G$, denoted by 
$sdim(G)$, is the cardinality of a strong metric basis of $G$.

A vertex cover of $G$ is a set of vertices $S \subseteq V(G)$ such that every 
edge of $G$ is incident to at-least one vertex of $S$. The minimum cardinality 
of a vertex cover of $G$ is called the vertex covering number of $G$ and is denoted 
by $\alpha(G)$. An independent set of $G$ is a set of vertices of $G$ such that no 
two of which are adjacent in $G$. The largest cardinality of an independent set of
$G$ is called the independence number of $G$ and is denoted by $\beta(G)$. 
	For an $n$-vertex graph $G$, it is known that 
		$S \subseteq V(G)$ is an independent set if and only if its complement 
		$\overline{S}$  is a vertex cover of 
		$G$ and hence $\alpha(G) + \beta(G)=n$.

Let $u,v\in V(G)$. We say that $u$ is maximally distant from $v$, denoted by 
$u{\rm MD}v$, if $d(v,w)\leq d(u,v)$ for every $w\in N(u)$. 
If $u{\rm MD} v$ and $v {\rm MD} u$ then $u$ and $v$ are called mutually maximally distant and we denote it by $u {\rm MMD} v$. 
The strong resolving graph $G_{SR}$ of $G$ is 
a graph with vertex set $V(G_{SR})=V(G)$ and $uv\in E(G_{SR})$ if 
and only if $u{\rm MMD}v$.

Oellermann and Peters-Fransen \cite{OP2007} further investigated the strong metric
dimension and showed that the problem of finding strong metric dimension can be
transformed into a more well-known problem. Furthermore, the authors  extended
the strong metric dimension of graphs to directed graphs.
We refer \cite{Eunjeong, kang2016, Jozef2012, Jozef, KYR2016} for 
further study on  strong metric dimension
of graphs. 
The relationship between strong metric
dimension of $G$ and vertex covering number
of $G_{SR}$ is given in the following theorem.
\begin{theorem}[Oellermann and Peters-Fransen \cite{OP2007}]
	\label{thm-vc-sd}
	For any connected graphs $G$,
	 $sdim(G)=\alpha(G_{SR})$. 
\end{theorem}

A generalized Jahangir graph $J(n,m)$ for $m \geq 3$ and $n\geq 2$, is a graph
on $nm + 1$ vertices consisting of a cycle $C_{nm}$
with one additional vertex which is adjacent to $m$ vertices of
$C_{nm}$ at distance $n$ to each other on $C_{nm}$
 (\cite{ABT2008}). For simplicity, the generalized Jahangir
  graph hereinafter referred to as Jahangir graph.
Tomescu and Javaid \cite{TJ2007} determined the
formula for the metric dimension of  
Jahangir graph $J(2,m)$.
In this paper, we determine the formula for 
strong metric dimension of Jahangir graph $J(n,m)$ for $m \geq 4$ and $n\geq 5$. We find 
strong resolving graph of  $J(n,m)$ and use
Theorem~\ref{thm-vc-sd} to determine the
formula for strong metric dimension of $J(n,m)$.

\section{Strong metric dimension of $J(n,m)$}
Let $u_i$, where $i= 1,\ldots, nm$, be the vertices 
of the cycle $C_{nm}$ in the graph $J(n,m)$ and 
$c$ be the vertex of  $J(n,m)$ which is adjacent
to $m$ vertices $\{u_{nk+1}\mid k=0,1,\ldots,m-1\}$ of the cycle $C_{nm}$.
Thus, the cycle $C_{nm}$ in $J(n,m)$ has two types of vertices: those that are adjacent to $c$ and those that are not 
adjacent to $c$. 
Let 
 $U_{1}=\{u_{nk+1}\mid k=0,1,\ldots,m-1\}\cup \{c\}$ 
 and
$U_{2}=V(J(n,m))\backslash U_{1}$. Then 
$d(c) = m$ and $d(x) = 3$ for each
$x\in U_1\setminus \{c\}$ and $d(x) = 2$ for each
$x\in U_2$. 
For each $k\in \{0,1,\ldots,m-1\}$, the cycle 
$cu_{nk+1}u_{nk+2}\ldots u_{n(k+1)+1}c$ is 
called an internal cycle of $J(n,m)$. Any two internal 
cycles in a Jahangir graph $J(n,m)$ have at-most one edge in common.
It is easy to see that 
$d(J(n,m)) = 2(\lfloor \frac{n}{2}\rfloor +1)$.
 A Jahangir graph $J(2,8)$ is depicted
 in Figure~\ref{j28}.
\begin{figure}[h]
	\centering
	\includegraphics[width=6cm]{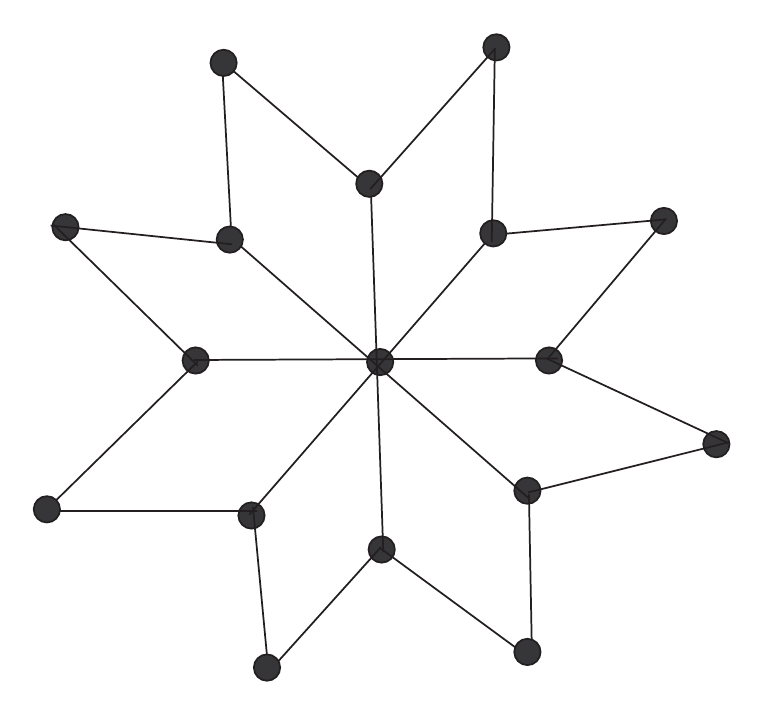}
	\caption{ Jahangir graph $J(2,8)$}	\label{j28}
\end{figure}	
\begin{lemma}
	If $m=3$ and $n\in\{2,3,4\}$ then 
	$sdim(J(n,m))=3$.
\end{lemma}
\begin{proof}
We first consider $J(2,3)$. Then  
$u_2{\rm MMD}u_5$, $u_4{\rm MMD}u_1$ and 
$u_6{\rm MMD}u_3$. Any other pair of vertices
of $J(2,3)$ does not contain mutually maximally distant vertices.
Therefore, the strong resolving graph of $J(2,3)$ consists of three copies of $K_{2}$ and an isolated vertex $c$ and thus $\alpha(J(2,3)_{SR})=3$.  

Secondly, we consider  $J(3,3)$. 
Then $u_2{\rm MMD}u_6$, $u_3{\rm MMD}u_8$ and $u_5{\rm MMD}u_9$. 
Any other pair of vertices
of $J(3,3)$ does not contain mutually maximally distant vertices.
Therefore, the strong resolving graph of $J(3,3)$
 consists of three copies of $K_{2}$ and four isolated vertices given by $u_1$, $u_4$, $u_7$
  and $c$. Thus $\alpha(J(3,3)_{SR})=3$. 

Finally, we consider $J(4,3)$. Then
 $u_3{\rm MMD}u_8$, $u_3{\rm MMD}u_{10}$,
  $u_7{\rm MMD}u_2$, $u_7{\rm MMD}u_{12}$,
   $u_{11}{\rm MMD}u_4$ and
    $u_{11}{\rm MMD}u_6$. 
    Any other pair of vertices
    of $J(4,3)$ does not contain mutually maximally distant vertices.
 Therefore, the strong resolving graph of $J(4,3)$ consists of three copies of 
 $P_{3}$ and four isolated vertices given by
  $u_1$, $u_5$, $u_9$ and $c$. Thus $\alpha(J(4,3)_{SR})=3$.  
  
  By Theorem~\ref{thm-vc-sd}, we have $sdim(J(n,3))=3$ for $n\in\{2,3,4\}$.     
\end{proof}
In rest of the paper, we consider the
graph $J(n,m)$ with $n\geq 5$ and $m\geq 4$. Lemmas~\ref{lemma-even} and \ref{lemma-odd} are easy to prove by looking at the structure of Jahangir graph $J(n,m)$.
\begin{lemma}
\label{lemma-even}
Assume that $n> 5$ is even and $m\geq 4$.
For $k,k'\in \{0,1,\ldots,m-1\}$, let  
$C=cu_{nk+1}u_{nk+2}\ldots u_{n(k+1)+1}c$ and  
$C'=cu_{nk'+1}u_{nk'+2}\ldots u_{n(k'+1)+1}c$ be two internal cycles of $J(n,m)$. 
Then we have the following:
\begin{description}
	\item[{\rm (a).}] Let  $C$ and $C'$ share exactly one edge in common and take $x\in V(C)$ and $y\in V(C')$. Set $k' = k+1$, where
	$k+1$ is an the integer modulo $m$. 
	Then $d(x,y) = n+1$ if and only if $\{x,y\}=\{u_{nk+\frac{n}{2}+1},u_{n(k+1)+\frac{n}{2}+2}\}$
	or $\{x,y\}=\{ u_{nk+\frac{n}{2}},u_{n(k+1)+\frac{n}{2}+1}\}$.
	\item[{\rm (b).}] Let  $C$ and $C'$ share no edge in common and take $x\in V(C)$ and $y\in V(C')$. 
	Then $d(x,y) = n+2$ if and only if $\{x,y\}=\{u_{nk+\frac{n}{2}+1}, u_{nk'+\frac{n}{2}+1}\}$.
	\item[{\rm (c).}] Let $u_{nk+i},u_{nk+j}\in V(C)\cap U_{2}$. Then 
	$d(u_{nk+i},u_{nk+j})=\frac{n}{2}+1$ if and only if $|i-j|=\frac{n}{2}+1$ .
\end{description}
\end{lemma}
\begin{lemma}
	\label{lemma-odd}
	Assume that $n\geq 5$ is odd and $m\geq 4$.
	For $k,k'\in \{0,1,\ldots,m-1\}$, let  
	$C=cu_{nk+1}u_{nk+2}\ldots u_{n(k+1)+1}c$ and  
	$C'=cu_{nk'+1}u_{nk'+2}\ldots u_{n(k'+1)+1}c$ be 
	two internal cycles of $J(n,m)$. 
	Then we have the following:
	\begin{description}
		\item[{\rm (a).}] 
		Let  $C$ and $C'$ share exactly one edge in common and take $x\in V(C)$ and $y\in V(C')$. 
		Then $d(x,y)=n+1$ if and only if  
		$\{x,y\} = \{u_{nk+\lfloor \frac{n}{2}\rfloor +1}, 
		u_{n(k+1)+\lfloor \frac{n}{2}\rfloor +2}\}$. 
		Also, $d(x,y) = n$  and both $x$ and $y$ do not belong to a
		diametrical path in $J(n,m)$ if and only if 
		$\{x,y\} = \{u_{nk+\lfloor \frac{n}{2}\rfloor }, u_{n(k+1)+\lfloor \frac{n}{2}\rfloor +1} \}$ or $\{x,y\} = \{u_{nk+\lfloor \frac{n}{2}\rfloor +2}, u_{n(k+1)+\lfloor \frac{n}{2}\rfloor +3} \}$.
		\item[{\rm (b).}]
		Let  $C$ and $C'$ share no edge in common and take $x\in V(C)$ and $y\in V(C')$. 
		Then $d(x,y) = n+1$ if and only if
		$x\in \{u_{nk+\lfloor \frac{n}{2}\rfloor +1}, u_{nk+\lfloor \frac{n}{2}\rfloor +2} \}$ and
		 $y\in \{u_{nk'+\lfloor \frac{n}{2}\rfloor +1},$  $ u_{nk'+\lfloor \frac{n}{2}\rfloor +2} \}$.
		\item[{\rm (c).}] 
		Let $x,y\in V(C)\cap U_{2}$. Then 
		$d(x,y)=\lfloor \frac{n}{2}\rfloor +1$ if and only if $\{x,y\}=\{u_{nk+i}, u_{nk+j}\}$
		with $|i-j|\in \{\frac{n}{2}+1, \frac{n}{2}+2\}$.
	\end{description}
\end{lemma}
The mutually maximally distant vertices of Jahangir graph $J(n,m)$
depend on the parity of $n$.
In next two theorems, we find
the mutually maximally distant vertices in the graph $J(n,m)$
when $n$ is even and when $n$ is odd.
\begin{theorem}
\label{thm-even}
	Assume that $n> 5$ is even and $m\geq 4$.
	For $k,k'\in \{0,1,\ldots,m-1\}$, let  
	$C=cu_{nk+1}u_{nk+2}\ldots u_{n(k+1)+1}c$ and  
	$C'=cu_{nk'+1}u_{nk'+2}\ldots u_{n(k'+1)+1}c$ be two internal cycles of $J(n,m)$. 
	Then we have the following:
	\begin{description}
		\item[{\rm (a).}] Let  $C$ and $C'$ share exactly one edge in common and take $x\in V(C)$ and $y\in V(C')$. 
		Then $x {\rm MMD} y$ if and only if $d(x,y) = n+1$.
		\item[{\rm (b).}] Let  $C$ and $C'$ share no edge in common and take $x\in V(C)$ and $y\in V(C')$. 
			Then $x {\rm MMD} y$ if and only if $d(x,y) = n+2$.
			\item[{\rm (c).}] Let $x,y\in V(C)$. Then 
				$x {\rm MMD} y$ if and only if $x,y\in U_{2}$ with $d(x,y)=\frac{n}{2}+1$.
	\end{description}
	\end{theorem}
\begin{proof}
		(a). By symmetry, we assume that $k'= k+1$, where $k+1$ is the integer modulo $m$. It is readily seen that 
		$d(u_{nk+\frac{n}{2}+1},u_{n(k+1)+\frac{n}{2}+2})=
		n+1= d(u_{nk+\frac{n}{2}},u_{n(k+1)+\frac{n}{2}+1})$ and it is the
		 largest distance between any two vertices of $C$ and $C'$.
		 All possible shortest paths from $u_{nk+\frac{n}{2}+1}$ to $u_{n(k+1)+\frac{n}{2}+2}$
		are given by:
		\begin{eqnarray}
		&& u_{nk+\frac{n}{2}+1}u_{nk+\frac{n}{2}+2} \ldots u_{n(k+1)}u_{n(k+1)+1}\ldots
		u_{n(k+1)+\frac{n}{2}+1}u_{n(k+1)+\frac{n}{2}+2}, \label{thm-paths1} \\
		&& u_{nk+\frac{n}{2}+1}u_{nk+\frac{n}{2}+2} \ldots u_{n(k+1)+1}cu_{n(k+2)+1}u_{n(k+2)}\ldots u_{n(k+1)+\frac{n}{2}+3}u_{n(k+1)+\frac{n}{2}+2}, \\
		&& u_{nk+\frac{n}{2}+1}u_{nk+\frac{n}{2}} \ldots
		u_{nk+1}cu_{n(k+2)+1}u_{n(k+2)}\ldots
		u_{n(k+1)+\frac{n}{2}+3}u_{n(k+1)+\frac{n}{2}+2}.
		\label{thm-paths2}
		\end{eqnarray}
		Also, all possible shortest paths 
		from $u_{nk+\frac{n}{2}}$ to $u_{n(k+1)+\frac{n}{2}+1}$ 
		are given by:
		\begin{eqnarray}
		&& u_{nk+\frac{n}{2}}u_{nk+\frac{n}{2}+1} \ldots 
		u_{n(k+1)}u_{n(k+1)+1}\ldots u_{n(k+1)+\frac{n}{2}}u_{n(k+1)+\frac{n}{2}+1},
		\label{thm-paths3}\\
		&& u_{nk+\frac{n}{2}}u_{nk+\frac{n}{2}-1} \ldots u_{nk+1}cu_{n(k+1)+1}u_{n(k+1)+2}\ldots
		u_{n(k+1)+\frac{n}{2}}u_{n(k+1)+\frac{n}{2}+1},\\
		&& u_{nk+\frac{n}{2}}u_{nk+\frac{n}{2}-1} \ldots u_{nk+1}cu_{n(k+2)+1}u_{n(k+2)}\ldots
		u_{n(k+1)+\frac{n}{2}+2}u_{n(k+1)+\frac{n}{2}+1}.
		\label{thm-paths4}
		\end{eqnarray}
		
		If $d(x,y)< n+1$ then $x$ and $y$ belong to one of the shortest paths given in  
		\eqref{thm-paths1}$\sim$\eqref{thm-paths4}. Thus $x$ and $y$ are not mutually maximally distant.
		
		Conversely, let $d(x,y) = n+1$. Then by
		Lemma~\ref{lemma-even} (a), we have  
		$\{x,y\}= \{u_{nk+\frac{n}{2}+1},u_{n(k+1)+\frac{n}{2}+2} \}$ 
		or $\{x,y\}=  \{u_{nk+\frac{n}{2}},u_{n(k+1)+\frac{n}{2}+1} \}$. 
		In either case, $x,y\in U_2$ and thus
		the neighbors of $x$ and $y$ in $J(n,m)$ belong to $V(C)\cup V(C')$.
		 This proves that $x {\rm MMD} y$. 
	
	(b). Note that 
	$d(u_{nk+\frac{n}{2}+1}, u_{nk'+\frac{n}{2}+1}) = n+2= d(J(n,m))$ 
	and all possible shortest paths
		   from $u_{nk+\frac{n}{2}+1}$ to $u_{nk'+\frac{n}{2}+1}$ are given by:
		\begin{eqnarray}
		&& u_{nk+\frac{n}{2}+1}u_{nk+\frac{n}{2}+2}\ldots u_{n(k+1)+1}cu_{nk'+1}u_{nk'+2}\ldots
		 u_{nk'+\frac{n}{2}}u_{nk'+\frac{n}{2}+1}, \label{thm-paths5}\\ 
		 && u_{nk+\frac{n}{2}+1}u_{nk+\frac{n}{2}+2} \ldots 
		 u_{n(k+1)+1}cu_{n(k'+1)+1}u_{n(k'+1)}\ldots u_{nk'+\frac{n}{2}}u_{nk'+\frac{n}{2}+1},\\
		 && u_{nk+\frac{n}{2}+1}u_{nk+\frac{n}{2}}
		 \ldots u_{nk+1}cu_{n(k'+1)+1}u_{n(k'+1)}\ldots u_{nk'+\frac{n}{2}}u_{nk'+\frac{n}{2}+1},\\
		 && u_{nk+\frac{n}{2}+1}u_{nk+\frac{n}{2}}\ldots u_{nk+1}cu_{n(k'+1)+1}u_{n(k'+1)}\ldots
		  u_{nk'+\frac{n}{2}}u_{nk'+\frac{n}{2}+1}.\label{thm-paths6}
		\end{eqnarray}
	
	If $d(x,y) < n+2 $ then $x$ and $y$ belong to one of the shortest paths given in 
	\eqref{thm-paths5}$\sim$\eqref{thm-paths6}. Thus  $x$ and $y$ are not mutually maximally distant.
	
	Conversely, assume that $d(x,y) = n+2$. Since $d(J(n,m)) = n+2$, 
	obviously 	$x{\rm MMD} y$.

(c).  
If $d(x,y)<\frac{n}{2}+1$ then one of the neighbors of $x$ on $C$, say $z$, does not belong to 
the shortest path from $x$ to $y$ and thus $d(z,y)=1+d(x,y)$. This shows that $x$ and $y$
are not mutually maximally distant. Also if $x\notin U_{2}$ then one of the neighbors of $x$, say $w$, 
does not belong to $C$ and
$d(w,y) = 1+ d(x,y)$. This shows that $x$ and $y$ are not mutually maximally distant.

Now we prove the converse. Since $C$ is an even cycle of length $n+2$, the largest distance between any two vertices of $C$ is 
$\frac{n}{2} + 1$. Let $x,y\in U_{2}$ with $d(x,y)=\frac{n}{2}+1$. Then all neighbors of $x$ and $y$ in $J(n,m)$
belong to $V(C)$ which implies that $x {\rm MMD} y$.
\end{proof}
For each $k\in \{0,1,\ldots, m-1\}$, Lemma~\ref{lemma-even} (a) and Theorem~\ref{thm-even} (a) imply that 
\begin{equation}\label{part-a}
u_{nk+\frac{n}{2}+1}u_{n(k+1)+\frac{n}{2}+2},
u_{nk+\frac{n}{2}+1}u_{n(k-1)+\frac{n}{2}} \in E(J(n,m)_{SR}),
\end{equation}
where $k-1$ and $k+1$ are integers modulo $m$. 
We denote by
${\cal A}_1$ the set of all edges given in \eqref{part-a} for each
$k\in \{0,1,\ldots, m-1\}$.

Similarly, for each 
$k,k'\in \{0,1,\ldots, m-1\}$ with $|k-k'|\notin \{1,m-1\}$, 
Lemma~\ref{lemma-even} (b) and Theorem~\ref{thm-even} (b) give
\begin{equation}\label{part-b}
u_{nk+\frac{n}{2}+1}u_{nk'+\frac{n}{2}+1}
\in E(J(n,m)_{SR}).
\end{equation}
Denote by
${\cal B}_1$ the set of all edges given in \eqref{part-b} for each 
$k,k'\in \{0,1,\ldots, m-1\}$ with $|k-k'|\notin \{1,m-1\}$.

Moreover, for each $k\in \{0,1,\ldots, m-1\}$,
Lemma~\ref{lemma-even} (c) and Theorem~\ref{thm-even} (c) give
\begin{equation}\label{part-c}
u_{nk+i}u_{nk+j}\in E(J(n,m)_{SR}),
\end{equation}
where $i\in \{2,\ldots,\frac{n}{2}-1\}$,
$j\in \{\frac{n}{2}+3,\ldots,n\}$ with
$j-i=\frac{n}{2}+1$.
We denote by
${\cal C}_1$ the set of all edges given in \eqref{part-c} for each
$k\in \{0,1,\ldots, m-1\}$.
Thus, $E(J(n,m))= {\cal A}_1\cup {\cal B}_1\cup {\cal C}_1$ when $n$ is even and $m\geq 4$.
\begin{lemma}
\label{lem-vc-even}
Assume that $n> 5$ is even and $m\geq 4$. Then
$\alpha(J(n,m)_{SR}) = m(\frac{n-2}{2})$.
\end{lemma}
\begin{proof}
We construct a vertex cover of $J(n,m)_{SR}$ with minimum cardinality.
	From \eqref{part-a}$\sim$\eqref{part-c}, we see that two pendent vertices
	$u_{n(k+1)+\frac{n}{2}+2}$ and 
	$u_{n(k-1)+\frac{n}{2}}$ are incident to $u_{nk+\frac{n}{2}+1}$ in  $J(n,m)_{SR}$, 
	for each $k\in \{0,1,\ldots, m-1\}$.
	Thus, the $m$ vertices $\{u_{nk+\frac{n}{2}+1}\mid k=0,1,\ldots, m-1\}$
	are minimum number of vertices to cover the edges
	 of ${\cal A}_1$. Let $S=\{u_{nk+\frac{n}{2}+1}\mid k=0,1,\ldots, m-1\}$.
	Then $S$ also covers
	 the edges of ${\cal B}_1$. 
	Moreover, the edges of ${\cal C}_1$ are
	$m(\frac{n}{2}-2)$ copies of $K_2$. 
	Therefore, minimum $m(\frac{n}{2}-2)$ vertices
	are required to cover the edges of ${\cal C}_1$. 
	To cover the edges of ${\cal C}_1$, we augment $S$ by taking 
	$S:= S\cup \{u_{nk+i}\mid 
	k=0,1,\ldots, m-1\mbox{ and } 
	i= 2,\ldots,\frac{n}{2}-1\}$.
	Thus $S$ is a vertex cover of 
	$J(n,m)_{SR}$ with minimum cardinality.
	The cardinality of $S$ is given by
	 $m + m(\frac{n}{2}-2) = m(\frac{n-2}{2})$.
		\end{proof}

We have the following main result when $n$ is even.
\begin{theorem}
Assume that $n> 5$ is even and $m\geq 4$. Then	$sdim(J(n,m)) =m(\frac{n-2}{2})$.
\end{theorem}
\begin{proof}
The proof follows from Theorem~\ref{thm-vc-sd} and
Lemma \ref{lem-vc-even}. 
\end{proof}


In the following theorem, we find mutually
 maximally distant vertices in 
 $J(n,m)$ when $n$ is odd. The proof is 
similar to the proof of Theorem~\ref{thm-even} and is omitted.
\begin{theorem}
	\label{thm-odd} 
	Assume that $n\geq 5$ is odd and $m\geq 4$.
	For $k,k'\in \{0,1,\ldots,m-1\}$, let  
	$C=cu_{nk+1}u_{nk+2}\ldots u_{n(k+1)+1}c$ and  
	$C'=cu_{nk'+1}u_{nk'+2}\ldots u_{n(k'+1)+1}c$ be two internal cycles of $J(n,m)$. 
	Then we have the following:
\begin{description}
	\item[{\rm (a).}] 
	Let  $C$ and $C'$ share exactly one edge in common and take $x\in V(C)$ and $y\in V(C')$. 
	If $d(x,y)=n+1$ then $x{\rm MMD}y$. Let $d(x,y)<n+1$. 
	Then  $x {\rm MMD} y$ if and only if $d(x,y) = n$  and both $x$ and $y$ do not belong to a
	 diametrical path in $J(n,m)$.
	\item[{\rm (b).}]
	Let  $C$ and $C'$ share no edge in common and take $x\in V(C)$ and $y\in V(C')$. 
	Then $x {\rm MMD} y$ if and only if $d(x,y) = n+1$.
	\item[{\rm (c).}]
	Let $x,y\in V(C)$. Then 
	$x {\rm MMD} y$ if and only if $x,y\in U_{2}$ with 
	$d(x,y)=\lfloor \frac{n}{2}\rfloor +1$.
	\end{description}
		\end{theorem}
	
	For each $k\in \{0,1,\ldots, m-1\}$, Lemma~\ref{lemma-odd} (a) 
	and Theorem~\ref{thm-odd} (a) imply that 
	the following edges belong to
	 $E(J(n,m)_{SR})$:
	 \begin{align}\label{part-a-odd}
	 \begin{split}
&	u_{nk+\lfloor \frac{n}{2}\rfloor}u_{n(k+1)+\lfloor \frac{n}{2}\rfloor+1},
	u_{nk+\lfloor \frac{n}{2}\rfloor + 1}u_{n(k-1)+\lfloor \frac{n}{2}\rfloor},\\
&	u_{nk+\lfloor \frac{n}{2}\rfloor + 1}u_{n(k+1)+\lfloor \frac{n}{2}\rfloor+2},
	u_{nk+\lfloor \frac{n}{2}\rfloor + 2}u_{n(k-1)+\lfloor \frac{n}{2}\rfloor+1},\\
&	u_{nk+\lfloor \frac{n}{2}\rfloor + 2}u_{n(k+1)+\lfloor \frac{n}{2}\rfloor+3},
	u_{nk+\lfloor \frac{n}{2}\rfloor + 3}u_{n(k-1)+\lfloor \frac{n}{2}\rfloor+2},
		\end{split}
		\end{align}
	where $k-1$ and $k+1$ are integers modulo $m$. Denote by
	${\cal A}_2$ the set of all edges given in \eqref{part-a-odd} for each
	$k\in \{0,1,\ldots, m-1\}$.
	
	Similarly, for each 
	$k,k'\in \{0,1,\ldots, m-1\}$ with $|k-k'|\notin \{1,m-1\}$, Lemma~\ref{lemma-odd} (b) and Theorem~\ref{thm-odd} (b) imply that
	the following edges belong to
	 $E(J(n,m)_{SR})$:
	\begin{align}\label{part-b-odd}
	\begin{split}
	& u_{nk+\lfloor \frac{n}{2}\rfloor+1}u_{nk'+\lfloor \frac{n}{2}\rfloor+1}, 
	u_{nk+\lfloor \frac{n}{2}\rfloor+1}u_{nk'+\lfloor \frac{n}{2}\rfloor+2},\\
	& u_{nk+\lfloor \frac{n}{2}\rfloor+2}u_{nk'+\lfloor \frac{n}{2}\rfloor+1},
	u_{nk+\lfloor \frac{n}{2}\rfloor+2}u_{nk'+\lfloor \frac{n}{2}\rfloor+2}.
	\end{split}
	\end{align}
	Denote by
	${\cal B}_2$ the set of all edges given in \eqref{part-b-odd} for each 
	$k,k'\in \{0,1,\ldots, m-1\}$ with $|k-k'|\notin \{1,m-1\}$.
	
	Moreover, for each $k\in \{0,1,\ldots, m-1\}$,
	Lemma~\ref{lemma-odd} (c) and Theorem~\ref{thm-odd} (c) imply that the following
	edge belongs to 
	 $E(J(n,m)_{SR})$:
	\begin{equation}\label{part-c-odd}
	u_{nk+i}u_{nk+j},
	\end{equation}
	where $i\in \{2,\ldots,\lfloor \frac{n}{2}\rfloor\}$,
	$j\in \{\lfloor \frac{n}{2}\rfloor+3,\ldots,n\}$ with
	$j-i\in \{\lfloor \frac{n}{2}\rfloor+1, \lfloor \frac{n}{2}\rfloor+2\}$.
	Denote by
	${\cal C}_2$ the set of all edges given in \eqref{part-c-odd} for each
	$k\in \{0,1,\ldots, m-1\}$.
	Thus $E(J(n,m)_{SR})={\cal A}_2\cup {\cal B}_2\cup {\cal C}_2$ when
	$n\geq 5$ is odd and $m\geq 4$.
	
	\begin{lemma}
	\label{lem-vc-odd}
	Assume that $n\geq 5$ is odd and $m\geq 4$. Then
	$\alpha(J(n,m)_{SR}) = m(\frac{n-1}{2})+m -3$.
	\end{lemma}
	\begin{proof}
	We construct a vertex cover of $J(n,m)_{SR}$
	 with minimum cardinality.
	For any
	$k,k'\in \{0,1,\ldots, m-1\}$ with $|k-k'|\notin \{1,m-1\}$, the vertices
	$u_{nk+\lfloor \frac{n}{2}\rfloor+1}$, 
$u_{nk+\lfloor \frac{n}{2}\rfloor+2}$, $u_{nk'+\lfloor \frac{n}{2}\rfloor+1}$ and 
$u_{nk'+\lfloor \frac{n}{2}\rfloor+2}$
form a cycle $C_4$ in $J(n,m)_{SR}$, where
$u_{nk+\lfloor \frac{n}{2}\rfloor+1}$
is not adjacent to
$u_{nk+\lfloor \frac{n}{2}\rfloor+2}$
and $u_{nk'+\lfloor \frac{n}{2}\rfloor+1}$ is not adjacent to
$u_{nk'+\lfloor \frac{n}{2}\rfloor+2}$ in $J(n,m)_{SR}$ (See \eqref{part-b-odd}).
Thus, to cover the edges of $C_4$, either
the pair of vertices $u_{nk+\lfloor \frac{n}{2}\rfloor+1},
u_{nk+\lfloor \frac{n}{2}\rfloor+2}$ or the the pair of vertices 
$u_{nk'+\lfloor \frac{n}{2}\rfloor+1}, u_{nk'+\lfloor \frac{n}{2}\rfloor+2}$
must be in any 
vertex cover of $J(n,m)_{SR}$.

Thus, for any $k\in \{0,1,\ldots, m-1\}$, 
the $2(m-2)$ vertices given by
$\{u_{n(k+l)+\lfloor \frac{n}{2}\rfloor+1}$, $
u_{n(k+l)+\lfloor \frac{n}{2}\rfloor+2}\mid 
l=0,1,\ldots,m-3\}$ 
are the minimum number of vertices required to
 cover all edges of ${\cal B}_2$. 
By symmetry, we take $k=0$ and let 
$S = \{u_{nl+\lfloor \frac{n}{2}\rfloor+1},
u_{nl+\lfloor \frac{n}{2}\rfloor+2}\mid 
l=0,1,\ldots,m-3\}$. 

We notice that all edges of ${\cal A}_2$ are also covered by $S$ except the following edges:
\begin{eqnarray}
&& u_{n(m-2)+\lfloor \frac{n}{2}\rfloor+1}
u_{n(m-1)+\lfloor \frac{n}{2}\rfloor+2}, 
u_{n(m-1)+\lfloor \frac{n}{2}\rfloor+2}
u_{\lfloor \frac{n}{2}\rfloor+3},
\label{uncovered3}\\
&& u_{n(m-2)+\lfloor \frac{n}{2}\rfloor+1}
u_{n(m-3)+\lfloor \frac{n}{2}\rfloor},
u_{n(m-2)+\lfloor \frac{n}{2}\rfloor+2}
u_{n(m-1)+\lfloor \frac{n}{2}\rfloor+3},
\label{uncovered1}\\
&& u_{n(m-1)+\lfloor \frac{n}{2}\rfloor+1}
u_{n(m-2)+\lfloor \frac{n}{2}\rfloor}.\label{uncovered4}
\end{eqnarray}
To cover the edge 
 $u_{n(m-2)+\lfloor \frac{n}{2}\rfloor+1}
 u_{n(m-1)+\lfloor \frac{n}{2}\rfloor+2}$
 given in \eqref{uncovered3},
 we 
augment the set $S$ by taking $S:= S\cup \{u_{n(m-1)+\lfloor \frac{n}{2}\rfloor+2}\}$.
Then the modified set $S$ also covers the edge
 $u_{n(m-1)+\lfloor \frac{n}{2}\rfloor+2}
u_{\lfloor \frac{n}{2}\rfloor+3}$ given
 in \eqref{uncovered3}.

Next, we find the minimum number of vertices which cover the edges of ${\cal C}_2$ and the leftover
edges given in  \eqref{uncovered1} and 
\eqref{uncovered4}.
For each $k\in \{0,1,\ldots, m-1\}$,
the edges given in \eqref{part-c-odd} form the path
$u_{nk+\lfloor \frac{n}{2}\rfloor+3}
u_{nk+2}
u_{nk+\lfloor \frac{n}{2}\rfloor+4}
u_{nk+3}
\ldots
u_{nk+n-1}
u_{nk+\lfloor \frac{n}{2}\rfloor-1}
u_{nk+n}
u_{nk+\lfloor \frac{n}{2}\rfloor}
$
of length $n-4$. To cover all edges of this path, 
we need 
minimum $\lfloor \frac{n-4}{2}\rfloor +1$ vertices. 
We further augment the set $S$ by taking 
 $S := S\cup \{u_{nk+i} \mid i\in \{2,\ldots,\lfloor \frac{n}{2}\rfloor\},
	k\in \{0,1,\ldots,m-2\}\}\cup \{u_{n(m-1)+i}\mid i\in\{\lfloor \frac{n}{2}\rfloor+3,\ldots, n\} \}$.
 Then the modified set $S$ 
 also covers the  
 edges of ${\cal C}_2$ and the three leftover edges given in
  \eqref{uncovered1}
 and \eqref{uncovered4}. Thus 
 $S$ is a vertex cover of $J(n,m)_{SR}$
 with minimum cardinality. The cardinality of $S$ is 
 given by $m(\frac{n-1}{2})+m -3$.
\end{proof}
Now we have the following main result
when $n$ is odd.  
\begin{theorem}
	Assume that $n\geq 5$ is odd and $m\geq 4$. Then
	$sdim(J(n,m)) = m(\frac{n-1}{2})+m -3$.
\end{theorem}
\begin{proof}
	The proof follows from Theorem~\ref{thm-vc-sd} and
	Lemma~\ref{lem-vc-odd}.
\end{proof}
	In the following example, we 
	first find strong resolving graph 
	of $J(6,5)$ using Theorem~\ref{thm-even}
	and then
	construct a vertex cover of $J(6,5)_{SR}$ with
	 minimum cardinality using the technique
	 shown in the proof of
	Lemma \ref{lem-vc-even}. This further gives
	strong metric dimension of $J(6,5)$.
	\begin{example}
		Consider the
		graph $J(6,5)$ shown in 
		Figure~\ref{fig4}(a). 
		Using  Theorem~\ref{thm-even},
		we have
		  $E(J(6,5)_{SR})={\cal A}_{1}
		  \cup {\cal  B}_{1}\cup {\cal C}_{1}$, where ${\cal A}_{1}$, 
		  ${\cal B}_{1}$ and ${\cal C}_{1}$ are given by:		
		  \begin{eqnarray*}
		 {\cal A}_{1} &=&
		\{u_{4}u_{11},u_{4}u_{27},u_{10}u_{3},u_{10}u_{17},u_{16}u_{9},
		u_{16}u_{23},u_{22}u_{15},u_{22}u_{19},\\
		&& u_{28}u_{21},u_{28}u_{5}\},\\
		 {\cal B}_{1} &=&
		\{u_{4}u_{16},u_{4}u_{22},u_{10}u_{22},u_{10}u_{28},u_{16}u_{28}\},\\
		{\cal C}_{1}&=&
		\{u_{2}u_{6},u_{8}u_{12},u_{14}u_{18},u_{20}u_{24},u_{26}u_{30}\}.
		\end{eqnarray*}
		The graph $J(6,5)_{SR}$ is shown in 
		Figure~\ref{fig4}(b).
		Using the technique
		shown in the proof of
		Lemma \ref{lem-vc-even} 
		the vertex cover of $J(6,5)_{SR}$
		 with minimum cardinality
		 is the set $S$ given by:
		 \begin{equation}
		 S =\{u_{4},u_{10},u_{16},u_{22},u_{28},
		 u_{2},u_{8},u_{14},u_{20},u_{26}\}. 
		 \end{equation}
		  Hence $sdim(J(6,5))=10$.
	\end{example}
\begin{figure}[H]
	\begin{center}
		{\includegraphics[width=14cm]{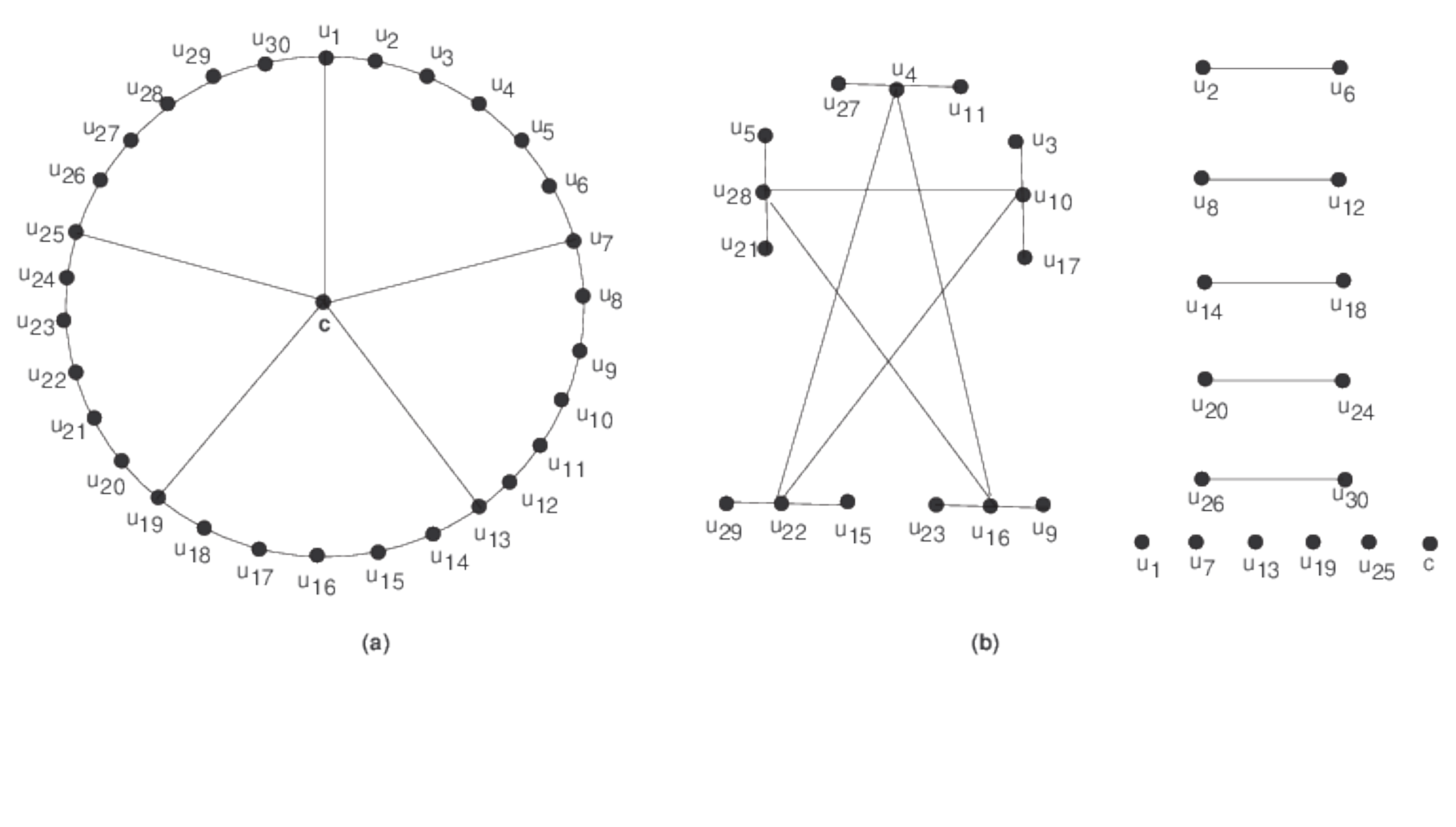}}
		\caption{(a) The graph $J(6,5)$, (b)
			The graph $J(6,5)_{SR}$}\label{fig4}
	\end{center}
\end{figure}
	
	In next example, we 
	find strong resolving graph 
	of $J(5,5)$ using Theorem~\ref{thm-odd}
	and then
	construct
	a vertex cover of $J(5,5)_{SR}$ with
	minimum cardinality using the technique
	shown in the proof of
	Lemma \ref{lem-vc-odd}. This further gives
	strong metric dimension of $J(5,5)$.
\begin{example}
	Consider the
	graph $J(5,5)$ shown in 
	Figure~\ref{fig5}(a). 
	Using  Theorem~\ref{thm-odd},
	we have
	$E(J(5,5)_{SR})={\cal A}_{2}
	\cup {\cal  B}_{2}\cup {\cal C}_{2}$, where ${\cal A}_{2}$, 
	${\cal B}_{2}$ and ${\cal C}_{2}$ are given by:
	\begin{eqnarray*}
	{\cal A}_{2} &=&\{u_{2}u_{8},u_{3}u_{22},u_{3}u_{9},u_{4}u_{23},u_{4}u_{10},u_{5}u_{24},u_{7}u_{13},u_{8}u_{14},u_{9}u_{15},u_{12}u_{18},\\
	&&
	u_{13}u_{19},
	u_{14}u_{20},u_{17}u_{23},u_{18}u_{24},u_{19}u_{25}\},\\
	{\cal B}_{2}&=&
		\{u_{3}u_{13},u_{3}u_{14},u_{4}u_{13},u_{4}u_{14},u_{3}u_{18},u_{3}u_{19},u_{4}u_{18},u_{4}u_{19},
		u_{8}u_{18},	
		u_{8}u_{19},\\
		&& u_{9}u_{18},
	u_{9}u_{19},u_{8}u_{23},u_{8}u_{24},u_{9}u_{23},u_{9}u_{24},u_{13}u_{23},u_{13}u_{24},u_{14}u_{23},u_{14}u_{24}\},\\
	{\cal C}_{2} &=&
	\{u_{2}u_{5},u_{7}u_{10},u_{12}u_{15},u_{17}u_{20},u_{22}u_{25}\}.
	\end{eqnarray*}
	The graph $J(5,5)_{SR}$ is shown in 
	Figure~\ref{fig5}(b).
	Using the technique
	shown in the proof of
	Lemma \ref{lem-vc-odd} 
	the vertex cover of $J(5,5)_{SR}$
	with minimum cardinality
	is the set $S$ given by:
	\begin{equation*}
	S = \{u_{3},u_{4},u_{8},u_{9},u_{13},u_{14},
	u_{12},u_{17},u_{24},u_{25},
	u_{2},u_{7}.
	\}
	\end{equation*}
	Hence $sdim(J(5,5))=12$.	
\end{example}
\begin{figure}[H]
	\begin{center}
		{\includegraphics[width=14cm]{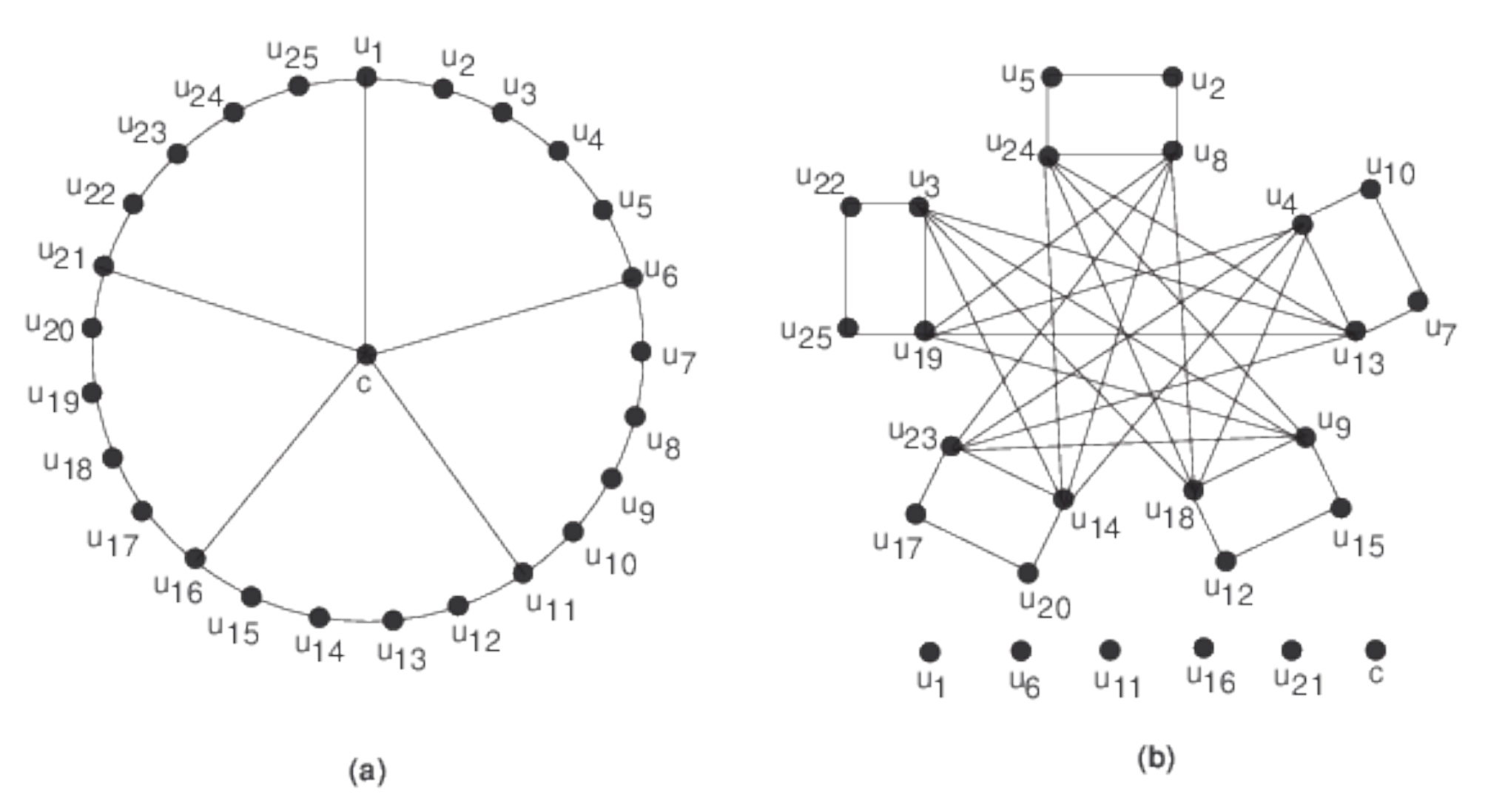}}
		\caption{(a) The graph $J(5,5)$, 
		(b) The graph $J(5,5)_{SR}$}\label{fig5}
	\end{center}
\end{figure}
	
\section*{Acknowledgment} 
Authors are thankful to the
Higher Education Commission of Pakistan for financial support to carry out this research under grant No. 20-3067/NRPU/R$\&$D/HEC/12.

\end{document}